\documentclass[12pt,reqno]{amsart}
\usepackage{amsmath}
\usepackage{amssymb, color}
\usepackage[left=3cm,top=3cm,right=3cm,bottom=3cm]{geometry}
\usepackage[usenames,dvipsnames]{xcolor}
\usepackage{graphicx,tikz}
\usetikzlibrary{decorations.pathreplacing}
\usepackage{epsfig}

\begin{document}
\newtheorem{theorem}{Theorem}[section]
\newtheorem{lemma}[theorem]{Lemma}
\newtheorem{claim}[theorem]{Claim}
\newtheorem{definition}[theorem]{Definition}
\newtheorem{conjecture}[theorem]{Conjecture}
\newtheorem{proposition}[theorem]{Proposition}
\newtheorem{algorithm}[theorem]{Algorithm}
\newtheorem{corollary}[theorem]{Corollary}
\newtheorem{observation}[theorem]{Observation}
\newtheorem{problem}[theorem]{Open Problem}
\newcommand{\noin}{\noindent}
\newcommand{\ind}{\indent}
\newcommand{\al}{\alpha}
\newcommand{\om}{\omega}
\newcommand{\pp}{\mathcal P}
\newcommand{\ppp}{\mathfrak P}
\newcommand{\R}{{\mathbb R}}
\newcommand{\N}{{\mathbb N}}
\newcommand\eps{\varepsilon}
\newcommand{\E}{\mathbb E}
\newcommand{\Bin}{\textrm{Bin}}
\newcommand{\Prob}{\mathbb{P}}
\newcommand{\pl}{\textrm{C}}
\newcommand{\dang}{\textrm{dang}}
\renewcommand{\labelenumi}{(\roman{enumi})}
\newcommand{\bc}{\bar c}
\newcommand{\G}{{\mathcal{G}}}
\newcommand{\Po}{{\mathcal{P}}}

\newcommand{\expect}[1]{\E \left [ #1 \right ]}
\newcommand{\floor}[1]{\left \lfloor #1 \right \rfloor}
\newcommand{\ceil}[1]{\left \lceil #1 \right \rceil}
\newcommand{\of}[1]{\left( #1 \right)}
\newcommand{\set}[1]{\left\{ #1 \right\}}
\newcommand{\angs}[1]{\left\langle #1 \right\rangle}
\newcommand{\sqbs}[1]{\left[ #1 \right]}
\newcommand{\sm}{\setminus}
\newcommand{\bfrac}[2]{\of{\frac{#1}{#2}}}
\renewcommand{\k}{\kappa}
\renewcommand{\l}{\ell}
\renewcommand{\b}{\beta}
\newcommand{\blue}[1]{{\color{blue} #1}}

\tikzset{
    position label/.style={
       below = 3pt,
       text height = 1.5ex,
       text depth = 1ex
    },
   brace/.style={
     decoration={brace},
     decorate
   }
}

\tikzset{
    position label/.style={
       below = 3pt,
       text height = 1.5ex,
       text depth = 1ex
    },
   bracem/.style={
     decoration={brace,mirror},
     decorate
   }
}

\title{The Total Acquisition Number of Random Geometric Graphs}

\author{Ewa Infeld}
\address{Department of Mathematics, Ryerson University, Toronto, ON, Canada, M5B 2K3}
\email{\texttt{evainfeld@ryerson.ca}}

\author{Dieter Mitsche}
\address{Universit\'{e} de Nice Sophia-Antipolis, Laboratoire J-A Dieudonn\'{e}, Parc Valrose, 06108 Nice cedex 02}
\email{\texttt{dmitsche@unice.fr}}

\author{Pawe\l{} Pra\l{}at}
\address{Department of Mathematics, Ryerson University, Toronto, ON, Canada
\and
The Fields Institute for Research in Mathematical Sciences, Toronto, ON, Canada}
\thanks{The third author is supported in part by NSERC and Ryerson University}
\email{\texttt{pralat@ryerson.ca}}

\begin{abstract}
Let $G$ be a graph in which each vertex initially has weight 1. In each step, the weight from a vertex $u$ to a neighbouring vertex $v$ can be moved, provided that the weight on $v$ is at least as large as the weight on $u$. The total acquisition number of $G$, denoted by $a_t(G)$, is the minimum cardinality of the set of vertices with positive weight at the end of the process. In this paper, we investigate random geometric graphs $\G(n,r)$ with $n$ vertices distributed u.a.r.\ in $[0,\sqrt{n}]^2$ and two vertices being adjacent if and only if their distance is at most $r$. We show that asymptotically almost surely $a_t(\G(n,r)) = \Theta( n / (r \lg r)^2)$ for the whole range of $r=r_n \ge 1$ such that $r \lg r \le \sqrt{n}$. By monotonicity, asymptotically almost surely $a_t(\G(n,r)) = \Theta(n)$ if $r < 1$, and $a_t(\G(n,r)) = \Theta(1)$ if $r \lg r > \sqrt{n}$.
\end{abstract}

\maketitle

\section{Introduction}

Gossiping and broadcasting are two well studied problems involving information dissemination in a group of individuals connected by a communication network~\cite{HHL}. In the gossip problem, each member has a unique piece of information which she would like to pass to everyone else. In the broadcast problem, there is a single piece  of information (starting at one member) which must be passed to every other member of the network. These problems have received attention from mathematicians as well as computer scientists due to their applications in distributed computing~\cite{BGRV}. Gossip and broadcast are respectively known as ``all-to-all'' and ``one-to-all'' communication problems. In this paper, we consider the problem of acquisition, which is a type of ``all-to-one'' problem. Suppose each vertex of a graph begins with a weight of 1 (this can be thought of as the piece of information starting at that vertex). A \textbf{total acquisition move} is a transfer of all the weight from a vertex $u$ onto a neighbouring vertex $v$, provided that immediately prior to the move, the weight on $v$ is at least the weight on $u$. Suppose a number of total acquisition moves are made until no such moves remain. Such a maximal sequence of moves is referred to as an \textbf{acquisition protocol}  and the vertices which retain positive weight after an acquisition protocol is called a \textbf{residual set}. Note that any residual set is necessarily an independent set. Given a graph $G$, we are interested in the minimum possible size of a residual set and refer to this number as the \textbf{total acquisition number of $G$}, denoted $a_t(G)$.  The restriction to total acquisition moves can be motivated by the so-called ``smaller to larger'' rule in disjoint set data structures. For example, in the UNION-FIND data structure with linked lists, when taking a union, the smaller list should always be appended to the longer list. This heuristic improves the amortized performance over sequences of union operations.  \vskip 0.1 in

\noindent \textbf{Example:} The weight of a vertex can at most double at every total acquisition move, and so a vertex with degree $d$ can carry at most weight $2^d$. (We will later use this fact in Observation~\ref{obs:min_degree}.) An acquisition protocol for a cycle $C_{4k}$ (for some $k \in \N$) that leaves a residual set of every fourth vertex is the best we can do; see Figure~\ref{fig:path}. Therefore, $a_t(C_{4k})=k.$

\begin{figure}[ht!]\centering
\begin{tikzpicture}
\draw (-0.75,0) -- (7.75,0);
\foreach \x in {0, 1, 2, 3, 4, 5, 6, 7}{
\filldraw[white] (\x,0) circle (6pt);
\draw (\x,0) circle (6pt);
\draw (\x,0) node {1};
}
\end{tikzpicture}\vskip 0.25 in

\begin{tikzpicture}
\draw (-0.75,0) -- (7.75,0);
\foreach \x in {0, 1, 2, 3, 4, 5, 6, 7}{
\filldraw[white] (\x,0) circle (6pt);
\draw (\x,0) circle (6pt);}
\foreach \x in {1, 2, 5, 6}{
\draw (\x,0) node {2};
}
\foreach \x in {0, 3, 4, 7}{
\draw (\x,0) node {0};
}
\draw (0.5,-0.3) node {$\rightarrow$};
\draw (4.5,-0.3) node {$\rightarrow$};
\draw (2.5,-0.3) node {$\leftarrow$};
\draw (6.5,-0.3) node {$\leftarrow$};
\end{tikzpicture}
\vskip 0.15 in

\begin{tikzpicture}
\draw (-0.75,0) -- (7.75,0);
\foreach \x in {0, 1, 2, 3, 4, 5, 6, 7}{
\filldraw[white] (\x,0) circle (6pt);
\draw (\x,0) circle (6pt);}
\foreach \x in {2,6}{
\draw (\x,0) node {4};
}
\foreach \x in {0, 1,  3, 4, 5, 7}{
\draw (\x,0) node {0};
}
\draw (1.5,-0.3) node {$\rightarrow$};
\draw (5.5,-0.3) node {$\rightarrow$};
\end{tikzpicture}
\caption{The total aquisition moves for a fragment of a cycle $C_{4k}$ that leave a residual set of size $a_t(C_{4k})=k$.} \label{fig:path}
\end{figure}
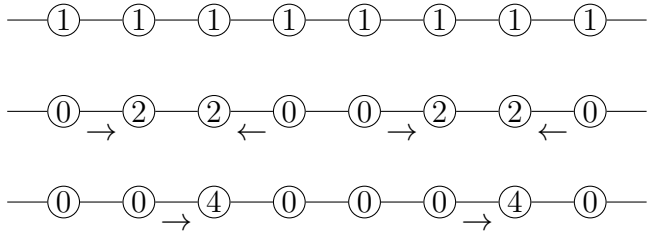

The parameter $a_t(G)$ was introduced by Lampert and Slater~\cite{LS} and subsequently studied in~\cite{SW, LPWWW}.  In~\cite{LS}, it was shown that $a_t(G)\le \floor{\frac{n+1}{3}}$ for any connected graph $G$ on $n$ vertices and that this bound is tight.  Slater and Wang~\cite{SW}, via a reduction to the three-dimension matching problem, showed that it is NP-complete to determine whether $a_t(G)=1$ for general graphs $G$. In LeSaulnier {\em et al.}~\cite{LPWWW}, various upper bounds on the acquisition number of trees were shown in terms of the diameter and the number of vertices, $n$. They also showed that $a_t(G) \le 32\log n\log\log n$ (here $\log n$ denotes the natural logarithm but throughout the paper we mostly use $\lg n$, the binary logarithm) for all graphs with diameter 2 and conjectured that the true bound is constant. For work on game variations of the parameter and variations where acquisition moves need not transfer the full weight of vertex, see~\cite{Wen, PWW, SW2}.

\medskip

Randomness often plays a part in the study of information dissemination problems, usually in the form of a random network or a randomized protocol, see {\em e.g.}~\cite{gos, FM, G95}. The total acquisition number of the \textbf{Erd\H{o}s-R\'{e}nyi-Gilbert random graph} $G(n,p)$ was recently studied in~\cite{BBDP}, where potential edges among $n$ vertices are added independently with probability $p$. In particular, LeSaulnier {\em et al.}~\cite{LPWWW} asked for the minimum value of $p=p_n$ such that $a_t(G(n,p)) = 1$ asymptotically almost surely (see below for a formal definition). In~\cite{BBDP} it was proved that $p = \frac{\lg n}{n} \approx 1.4427 \ \frac{\log n}{n}$ is a sharp threshold for this property. Moreover, it was also proved that almost all trees $T$ satisfy $a_t(T) = \Theta(n)$, confirming a conjecture of West. Another way randomness can come into the picture is when initial weights are generated at random. This direction, in particular the case where vertex weights are initially assigned according to independent Poisson distributions of intensity $1$, was recently considered in~\cite{GKKPZ}.

\medskip

In this note we consider the \textbf{random geometric graph} $\G(\mathcal{X}_n,r_n)$, where (i) $\mathcal{X}_n$ is a set of $n$ points located independently uniformly at random in $[0,\sqrt{n}]^2$, (ii) $(r_n)_{n \ge 1}$ is a sequence of positive real integers, and (iii) for $\mathcal{X} \subseteq \mathbb{R}^2$ and $r > 0$, the graph $\G(\mathcal{X}, r)$ is defined to have vertex set $\mathcal{X}$, with two vertices connected by an edge if and only if their spatial locations are at Euclidean distance at most $r$ from each other. As typical in random graph theory, we shall consider only asymptotic properties of $\G(\mathcal{X}_n,r_n)$ as $n\rightarrow \infty$. We will therefore write $r=r_n$, identifying vertices with their spatial locations and defining $\G(n,r)$ as the graph with vertex set $[n]=\{1,2,\dots, n\}$ corresponding to $n$ locations chosen independently uniformly at random in $[0,\sqrt{n}]^2$ and a pair of vertices within Euclidean distance $r$ appears as an edge. For more details see, for example, the monograph~\cite{pen}.

Finally, we say that an event in a probability space holds \textbf{asymptotically almost surely} (\textbf{a.a.s.}), if its probability tends to one as $n$ goes to infinity.

\medskip

We are going to show the following result. 

\begin{theorem}\label{thm:main}
Let $r=r_n$ be any positive real number. Then, a.a.s.\ $a_t \left( \G(n,r) \right) = \Theta( f_n )$, where
$$
f_n = 
\begin{cases}
n & \text{ if } r < 1, \\
\frac {n}{(r \lg r)^{2}} & \text{ if } r \ge 1 \text { and } r \lg r \le \sqrt{n},\\
1 & \text{ if } r \lg r > \sqrt{n}.
\end{cases}
$$
\end{theorem}

\section{Lower Bound}

Let us start with the following simple but useful observation. 

\begin{observation}\label{obs:min_degree}
Let $G=(V,E)$ be a graph. If $v \in V$ is to acquire weight $w$ (at any time during the process of moving weights around), then $deg(v)$, the degree of $v$, is at least $\lg w$. Moreover, all vertices that contributed to the weight of $w$ (at this point of the process) are at graph distance at most $\lg w$ from $v$. 
\end{observation}
\begin{proof}
Note that during each total acquisition move, when weight is shifted onto $v$ from some neighbouring vertex, the weight of $v$ can at most double. Thus, $v$ can only ever acquire $1 + 2 + \ldots + 2^{\deg(v)-1}$, in addition to the $1$ it starts with, and so $v$ can acquire at most weight $2^{\deg(v)}$. To see the second part, suppose that some vertex $u_0$ moved the initial weight of $1$ it started with to $v$ through the path $(u_0, u_1, \ldots, u_{k-1}, u_k=v)$. It is easy to see that after $u_{i-1}$ transfers its weight onto $u_i$, $u_i$ has weight at least $2^i$. So if $u_0$ contributed to the weight of $w$, $u_0$ must be at graph distance at most $\lg w$ from $v$. The proof of the observation is finished.
\end{proof}

We will also use the following consequence of Chernoff's bound (see, for example,~\cite{JLR} and~\cite{AS}).

\begin{theorem}[\textbf{Chernoff's Bound}] 
\par \noindent 
\begin{itemize}
\item [(i)] If $X$ is a Binomial random variable with expectation $\mu$, and $0<\delta<1$, then $$\Pr[X < (1-\delta)\mu] \le \exp \left( -\frac{\delta^2 \mu}{2} \right),$$ and if $\delta > 0$,
\[\Pr\sqbs{X > (1+\delta)\mu} \le \exp\of{-\frac{\delta^2 \mu}{2+\delta}}.\]
\item [(ii)] If $X$ is a Poisson random variable with expectation $\mu$, and $0 < \varepsilon < 1$, then 
$$
Pr\sqbs{X > (1+\varepsilon)\mu}\le \exp \left( -\frac{\varepsilon^2 \mu}{2} \right),
$$
and if $\varepsilon > 0$,
$$
\Pr\sqbs{X > (1+\varepsilon)\mu} \le \left(\frac{e^{\varepsilon}}{(1+\varepsilon)^{1+\varepsilon}}\right)^{\mu}. 
$$
\end{itemize}
In particular, for $X$ being a Poisson or a Binomial random variable with expectation $\mu$ and for $0 < \varepsilon < 1$, we have
$$
Pr\sqbs{|X-\mu| > \varepsilon\mu} \le 2\left( -\frac{\varepsilon^2 \mu}{3} \right).
$$
\end{theorem}

Now we are ready to prove the lower bound. First we concentrate on dense graphs for which, in fact, we show a stronger result that no vertex can acquire large weight a.a.s.

\begin{theorem}\label{thm:lower1}
Suppose that $r = r_n \ge c \sqrt{\lg n} / \lg \lg n$ for some sufficiently large $c \in \R$, and consider any acquisition protocol on $\G(n,r)$. Then, a.a.s.\ each vertex in the residual set acquires $O( (r \lg r)^2 )$ weight. As a result, a.a.s.\
$$
a_t \left( \G(n,r) \right) = \Omega \left( \frac {n}{(r \lg r)^{2}} \right).
$$
\end{theorem}
\begin{proof}
Let $\ell = 2 \lg r + 2 \lg \lg r + \lg (8 \pi)$. For a contradiction, suppose that at some point of the process some vertex $v$ acquires weight $w \ge 2^\ell = 8 \pi (r \lg r)^2$. Since one total acquisition move corresponding to transferring all the weight from some neighbour of $v$ onto $v$, increases the weight on $v$ by a factor of at most 2, we may assume that $w < 2^{\ell+1}$. It follows from Observation~\ref{obs:min_degree} that all vertices contributing to the weight of $w$ are at graph distance at most $\ell+1$ from $v$ (and so at Euclidean distance at most $(\ell+1)r$). The desired contradiction will be obtained if no vertex has at least $2^\ell$ vertices (including the vertex itself) at Euclidean distance at most $(\ell+1)r$. 

The remaining part is a simple consequence of Chernoff's bound and the union bound over all vertices. For a given vertex $v$, the number of vertices at Euclidean distance at most $(\ell+1)r$ is a random variable $Y$ that is stochastically bounded from above by the random variable $X \sim \Bin(n-1, \pi (\ell+1)^2 r^2 / n)$ with $\E [X] \sim \pi \ell^2 r^2 \sim 4 \pi (r \lg r)^2$. (Note that $Y=X$ if $v$ is at distance at least $(\ell+1)r$ from the boundary; otherwise, $Y \le X$.) It follows from Chernoff's bound that 
\begin{eqnarray*}
\Prob (Y \ge 2^\ell) &\le& \Prob \Big(X \ge (2+o(1)) \E[X] \Big) \le \exp \Big( -(1/3+o(1)) \E[X] \Big) \\
&\le& \exp \Big( -(4 \pi /3+o(1)) (r \lg r)^2 \Big) \le \exp \Big( -(\pi c^2 /3+o(1)) \lg n \Big) \\
&=& o(1/n),
\end{eqnarray*}
provided that $c$ is large enough. The conclusion follows from the union bound over all $n$ vertices of $\G(n,r)$.
\end{proof}

In order to simplify the proof of the theorem for sparser graphs we will make use of a technique known as de-Poissonization, which has many applications in geometric probability (see~\cite{pen} for a detailed account of the subject). Here we only sketch it.

Consider the following related model of a random geometric graph. Let $V=V'$, where $V'$ is a set obtained as a homogeneous Poisson point process of intensity $1$ in $[0,\sqrt{n}]^2$. In other words, $V'$ consists of $N$ points in the square $[0,\sqrt{n}]^2$ chosen independently and uniformly at random, where $N$ is a Poisson random variable of mean $n$. Exactly as we did for the model $\G(n,r)$, again identifying vertices with their spatial locations, we connect by an edge $u$ and $v$ in $V'$ if the Euclidean distance between them is at most $r$. We denote this new model by $\Po(n,r)$.
%Using the standard technique of de-Poissonization we know that any event holding in $\RGuv$ with probability at least $1-o(f_n)$ also it must hold in $\RG$ with probability at least $1-o(f_n \sqrt n)$.

The main advantage of defining $V'$ as a Poisson point process is motivated by the following two properties: the number of vertices of $V'$ that lie in any region $A\subseteq [0,\sqrt{n}]^2$ of area $a$ has a Poisson distribution with mean $a$, and the number of vertices of $V'$ in disjoint regions of $[0,\sqrt{n}]^2$ are independently distributed. Moreover, by conditioning $\Po(n,r)$ upon the event $N=n$, we recover the original distribution of $\G(n,r)$. Therefore, since $\Pr(N=n)=\Theta(1/\sqrt n)$, any event holding in $\Po(n,r)$ with probability at least $1-o(f_n)$ must hold in $\G(n,r)$ with probability at least $1-o(f_n \sqrt n)$.

\medskip

Now, let us come back to our problem. For sparser graphs we cannot guarantee that no vertex acquires large weight a.a.s.\ but a lower bound of the same order holds.

\begin{theorem}\label{thm:lower2}
Suppose that $r = r_n \ge c$ for some sufficiently large $c \in \R$. Then, a.a.s.\ 
$$
a_t \left( \G(n,r) \right) = \Omega \left( \frac {n}{(r \lg r)^{2}} \right).
$$
\end{theorem}
\begin{proof}
Since Theorem~\ref{thm:lower1} applies to dense graphs, we may assume here that $r = O(\sqrt{\lg n} / \lg \lg n)$ (in particular, $r \lg r = o(\sqrt{n})$). Tessellate $[0,\sqrt{n}]^2$ into $\lfloor \sqrt{n} / (20 r \lg r) \rfloor^2$ squares, each one of side length $(20+o(1)) r \lg r$. Consider the unit circle centered on the center of each square and call it the \emph{center circle}. We say that a given square is \emph{dangerous} if the corresponding center circle contains at least one vertex and the total number of vertices contained in the square is less than $1200 (r \lg r)^2$.

Consider any acquisition protocol. First, let us show that at least one vertex from each dangerous square must belong to the residual set. Let $u_0$ be a vertex inside the corresponding center circle. For a contradiction, suppose that the square has no vertex in the residual set. In particular, it means that $u_0$ moved the initial weight of 1 it started with onto some vertex outside the square through some path $(u_0, u_1, \ldots, u_k)$. Note that the Euclidean distance between $u_0$ and the border of the square (and so also $u_k$) is at least $(20+o(1)) r \lg r / 2 - 1 \ge 9 r \lg r$, provided that $c$ is large enough, and so $k \ge 9 \lg r$. 

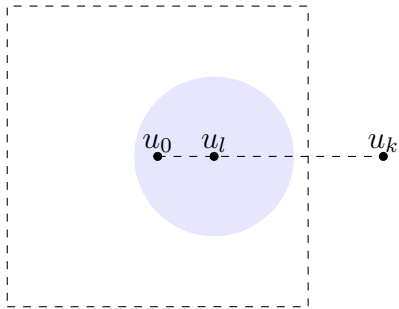
\begin{figure}[ht!]\centering
\begin{tikzpicture}[scale=0.5]
\filldraw[blue!10] (1.5,0) circle (60pt);
\foreach \x in {(0,0), (1.5,0), (6,0)}{
\filldraw \x circle (3pt);}
\draw[dashed] (0,0) -- (6,0);
\draw[dashed] (-4,-4) rectangle (4,4);
\draw (0,0.35) node {$u_0$};
\draw (1.5,0.35) node {$u_l$};
\draw (6,0.35) node {$u_k$};
\end{tikzpicture}
\caption{Residual sets contain at least one vertex from each dangerous square.} \label{fig:dangerous_square}
\end{figure}

Consider the vertex $u_\ell$ on this path, where $\ell = \lfloor 4 \lg r \rfloor \ge 3 \lg r$, provided $c$ is large enough; see Figure~\ref{fig:dangerous_square}. Right after $u_{\ell-1}$ transferred all the weight onto $u_\ell$, $u_\ell$ had weight at least $2^{\ell} \ge r^3 > 1200 (r \lg r)^2$, provided $c$ is large enough. As argued in the proof of the previous theorem, at some point of the process $u_\ell$ must have acquired weight $w$ satisfying $2^{\ell} \le w < 2^{\ell+1}$. Observation~\ref{obs:min_degree} implies that all vertices contributing to the weight of $w$ are at Euclidean distance at most $(\ell+1) r$ from $v$ and so inside the square (as always, provided $c$ is large enough). However, dangerous squares contain less than $1200(r\lg r)^2$ vertices, and so we get a contradiction. The desired claim holds.

Showing that a.a.s.\ a positive fraction of the squares is dangerous is straightforward. In $\Po(n,r)$, the probability that the center circle contains no vertex is $\exp(-\pi) \le 1/3$. On the other hand, the number of vertices falling into the square is a Poisson random variable $X$ with expectation $\mu \sim 400 (r \lg r)^2$. By Chernoff's bound applied with $\eps = e-1$,
$$
\Prob (X \ge e \mu) \le \left(\frac{e^{e-1}}{(1+(e-1))^e}\right)^{\mu}= \exp(-\mu).
$$
Hence, we get
$$
\Prob (X \ge 1200 (r \lg r)^2 ) \le \Prob (X \ge e\mu ) \le \exp(- \mu) \le 1/3,
$$
provided $c$ is large enough. Hence the expected number of dangerous squares is at least $(1/3) (1/400+o(1)) n / (r \lg r)^2 \gg \lg n \to \infty$. By Chernoff bounds, with probability at least $1-o(n^{-1/2})$, the number of dangerous squares in $\Po(n,r)$ is at least $(1/2500) n / (r \lg r)^2$. By the de-Poissonization argument mentioned before this proof, the number of dangerous squares in $\G(n,r)$ is a.a.s.\ also at least $(1/2500) n / (r \lg r)^2$, and the proof of the theorem is finished.
\end{proof}

The only range of $r=r_n$ not covered by the two theorems is when $r < c$ for $c$ as in Theorem~\ref{thm:lower2}. However, in such a situation a.a.s.\ there are $\Omega(n)$ isolated vertices which clearly remain in the residual set. Moreover, if $r$ is such that $r \lg r > \sqrt{n}$, then the trivial lower bound $\Omega(1)$ applies. The lower bound in the main theorem holds for the whole range of $r$.

\section{Upper Bound}

As in the previous section, let us start with a simple, deterministic observation that turns out to be useful in showing an upper bound. Before we state it, let us define a family of rooted trees as follows. Let $\hat{T}_0$ be a rooted tree consisting of a single vertex $v$ (the root of $\hat{T}_0$). For $i \in \N$, we define $\hat{T}_i$ recursively: the root $v$ of $\hat{T}_i$ has $i$ children that are roots of trees $\hat{T}_0, \hat{T}_1, \ldots, \hat{T}_{i-1}$; see Figure~\ref{fig:Tree_T}. 

\begin{figure}[ht!]\centering
\begin{tikzpicture}[scale=0.75]
\foreach \x in {(0,0), (-2,-1), (-1,-1), (0,-1), (-2,-2), (-1,-2), (0,-2), (0,-3)}{
\filldraw \x circle (2pt);
}
\foreach \x in {(-2,-1), (-1,-1), (0,-1),(1.1,-0.6)}{
\draw (0,0) -- \x;
}
\draw (0.75,-1) node {$\dots$};
\draw (-1,-1) -- (-2,-2);
\draw (-1,-2) -- (0,-1) -- (0,-3);
\draw (1.8,-1) node {$\hat{T}_{i-1}$};
\end{tikzpicture}\hskip 1 in
\begin{tikzpicture}[scale=0.75]
\foreach \x in {(0,0)}{
\filldraw \x circle (2pt);
}
\draw (-2,-1.1) node {$\hat{T}_0$};
\draw (-1,-1.1) node {$\hat{T}_1$};
\draw (0,-1.1) node {$\hat{T}_2$};
\foreach \x in {(-1.8,-0.8), (-0.8,-0.8), (0,-0.75),(1.1,-0.6)}{
\draw (0,0) -- \x;
}
\draw (0.77,-1.2) node {$\dots$};
\draw (1.8,-1.1) node {$\hat{T}_{i-1}$};
\draw[white] (0,-3) node {.};
\end{tikzpicture}
\caption{The tree $\hat{T}_i$.} \label{fig:Tree_T}
\end{figure}
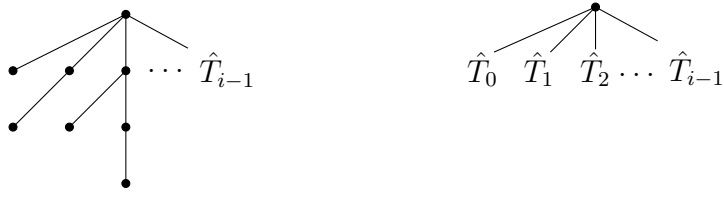

Clearly, $\hat{T}_i$ has $2^i$ vertices and depth $i$. Moreover, it is straightforward to see that vertices of $\hat{T}_i$ can move their initial weight of 1 to the root $v$ (in particular, $a_t(\hat{T}_i)=1$): indeed, this clearly holds for $i=0$ so suppose that it holds inductively up to $i-1$. Then, since all children of the root of $\hat{T}_i$ can send all their accumulated weight to the root of $\hat{T}_i$ (starting from the smallest subtree), this also holds for $i$. This, in particular, shows that Observation~\ref{obs:min_degree} is tight. 

As showed in the previous section, the main bottleneck that prevents us from moving a large weight to some vertex in $\G(n,r)$ is that there are simply not enough vertices in the Euclidean neighborhood of a vertex. If we want to match the lower bound, then rooted trees induced by the acquisition protocol must be as deep as possible in order to access vertices that are in a Euclidean sense far away from the corresponding roots. It turns out that trees $\hat{T}_i$ from the family we just introduced are efficient from that perspective. However, we cannot guarantee that the vertex set of $\G(n,r)$ can be partitioned in such a way that each set has some tree from the family as a spanning subgraph. Fortunately, it is easy to ``trim'' $\hat{T}_i$ to get a tree on $n < 2^i$ vertices that can shift all of its initial weight to the root.

\begin{observation}\label{obs:trees}
For any $d \in \N \cup \{0\}$ and $n \le 2^d$, $\hat{T}_d$ contains a rooted sub-tree $T$ on $n$ vertices such that $a_t(T)=1$. Moreover, the number of vertices at distance $\ell$ ($0 \le \ell \le d$) from the root of $T$ is at most $\binom{d}{\ell}$. 
\end{observation}
\begin{proof}
In order to obtain the desired tree $T$ on $n$ vertices, we trim $\hat{T}_d$ by cutting some of its branches (from largest to smallest, level by level). We may assume that $n \ge 2$; otherwise, the statement trivially holds.

Since we will be trimming the tree recursively, let us concentrate on $v$, the root of $\hat{T}_d$, and $d$ branches attached to it. Our goal is to get a tree rooted at $v$ that has $n \ge 2$ vertices. Let $k_0$ be the largest integer $k$ such that 
$$
1 + \Big( 1 + 2 + 4 + \ldots 2^k \Big) = 2^{k+1} \le n;
$$
that is, $k_0 = \lfloor \lg n \rfloor - 1$ (note that $k_0 \ge 0$ as $n \ge 2$ and that $k_0 \le d-1$ as $n \le 2^d$). We leave the branches inducing the trees $\hat{T}_0, \hat{T}_1, \ldots, \hat{T}_{k_0}$ untouched. We trim the branches inducing the trees $\hat{T}_{k_0+2}, \hat{T}_{k_0+3}, \ldots, \hat{T}_{d}$ completely (note that possibly $k_0 = d-1$ in which case we trim nothing). Finally, we would like to carefully trim the branch inducing the tree $\hat{T}_{k_0+1}$ so that the number of vertices it contains is precisely $n - 2^{k_0+1}$. If $n - 2^{k_0+1}$ is equal to 0 or 1, then we trim the whole branch or leave just the root of this branch, respectively. Otherwise, we recursively trim the branch as above. It is straightforward to see that all vertices of $T$ can move their initial weight of 1 to the root of $T$ which, in particular, implies that $a_t(T)=1$, thus proving the first part.

In order to show the second part, it is enough to prove the desired property for $\hat{T}_d$ (since $T$ is a sub-tree of $\hat{T}_d$). We prove it by (strong) induction on $d$; clearly, the statement holds for $d=0$ and $\ell=0$. Let $d_0 \in \N$ and suppose inductively that the property holds for all $d$ such that $0 \le d \le d_0-1$. The claim clearly holds for $\ell=0$. We count the number of grandchildren at distance $\ell$ (for any $1 \le \ell \le d_0$) from the root $v$ by considering grandchildren at distance $\ell-1$ from each child of $v$. By the recursive construction of $\hat{T}_d$ we get that the number of vertices at distance $\ell$ from $v$ is $\sum_{k=\ell-1}^{d_0-1} \binom{k}{\ell-1} = \binom{d_0}{\ell}$ (this equality is well-known and can be easily proven by induction). The proof of the observation is finished.
\end{proof}

\bigskip

Before we are ready to state the next result, we need to introduce a few definitions. Let $c, \eps \in (0,1)$ be any constants, arbitrarily small. Suppose that we are given a function $r = r_n$ such that $r \lg r \le \sqrt{n}$ and $r \ge C$ for some large constant $C=C(c,\eps)$ that will be determined soon. Let $k = \lceil \sqrt{n}/(c r \lg r) \rceil$ and tessellate $[0,\sqrt{n}]$ into $k^2$ \textbf{large squares}, each one of side length $x r \lg r$, where $x = \sqrt{n}/(k r \lg r)$. Clearly, $c/2 \le x \le c$ (the lower bound follows as $c r \lg r \le \sqrt{n}$) and $x \sim c$, provided $r \lg r = o(\sqrt{n})$. Now, let $\ell = 20 \lceil x r \lg r / (20 c r) \rceil = 20 \lceil x \lg r/(20 c) \rceil$ and tessellate each large square into $\ell^2$ \textbf{small squares}, each one of side length $yr$, where $y =x \lg r/\ell$; see Figure~\ref{fig:tessellate}. Clearly, $c/2 \le y \le c$ (the lower bound follows assuming that $C$ is large enough which we may) and $y \sim c$, provided $r = r_n \to \infty$ as $n \to \infty$.

\begin{figure}[ht!]\centering
\begin{tikzpicture}[scale=0.75]
\draw[thick] (1.5,1.5) rectangle (7.5,7.5);
\foreach \x in {1.5,3,4.5,6}{
\foreach \y in {0.375,0.75,1.125}{
\draw[dashed] (1.2,\x+\y) -- (7.8,\x+\y);
\draw[dashed] (\x+\y,1.2) -- (\x+\y,7.6);
}}
\foreach \x in {3,4.5,6}{
\draw (1.2,\x) -- (7.8,\x);
\draw (\x,1.2) -- (\x,7.8);}
\draw (2.25,8.1) node {\large{$\frac{\sqrt{n}}{k}$}};
\draw (0.8,7.3) node {\large{$\frac{\sqrt{n}}{kl}$}};
\end{tikzpicture}
\caption{We tessellate $[0,\sqrt{n}]^2$ into $k^2$ large squares, and each large square is tessellated into $l^2$ small squares; $k=\lceil \sqrt{n}/(cr\lg r)\rceil$ and $l=20\lceil x\lg r/(20c)\rceil$.} \label{fig:tessellate}
\end{figure}
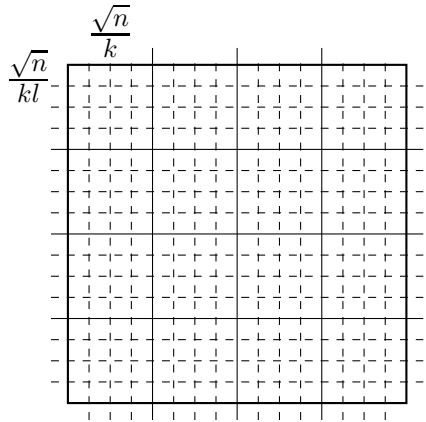

We say that a small square is \textbf{good} if the number of vertices it contains is between $(1-\eps) (yr)^2$ and $(1+\eps) (yr)^2$; otherwise, it is \textbf{bad}. Moreover, we say that a large square is \textbf{good} if all small squares it contains are good and the following properties hold (otherwise, it is \textbf{bad}):
\begin{itemize}
\item [(a)] no vertex lies on the border of the large square nor on its two diagonals,
\item [(b)] no two vertices lie on any line parallel to any base of the large square,
\item [(c)] no two vertices lie on any line passing the center of the large square.
\end{itemize}

Now we are ready to state the following crucial observation.

\begin{theorem}\label{thm:distr_good_squares}
For any pair $c, \eps \in (0,1)$ of constants, there exists a constant $C=C(c,\eps)$ such that the following two properties hold a.a.s.\ for $\G(n,r)$. 
\begin{itemize}
\item [(i)] All large squares are good, provided that $r \ge C \sqrt{\log n}$.
\item [(ii)] The number of large squares that are bad is at most $n/(r^2 \lg^5 r)$, provided that $r \ge C$.
\end{itemize}
\end{theorem}
\begin{proof}
Properties (a)-(c) on the distribution of the vertices hold with probability 1 for all large squares. Hence, we need to concentrate on showing that small squares are good.

For part (i), consider any small square in $\G(n,r)$. The number of vertices in such a square follows a binomial random variable $X \sim \Bin(n,(yr)^2/n)$ with $\E[X] =(yn)^2$. It follows immediately from Chernoff's bound that the probability of the square being bad can be estimated as follows:
$$
\Prob \left( |X - (yr)^2| \ge \eps (yr)^2 \right) \le 2 \exp \left( - \frac {\eps^2 (yr)^2}{3} \right) \le 2/n^2 = o(1/n),
$$
provided that $C \ge \sqrt{6}/(c\eps)$. Hence, since there are in total $O(n)$ small squares appearing in large squares, the expected number of such small squares is $o(1)$, and the conclusion follows from the first moment method.

For part (ii), consider any small square in $\Po(n,r)$. As before, let $X \sim Po( (yr)^2 )$ be the random variable counting the number of vertices in the square. By Chernoff's bound, for the probability of the square being bad we have
$$
\Prob \left( |X - (yr)^2| \ge \eps (yr)^2 \right) \le 2 \exp \left( - \frac {\eps^2 (yr)^2}{3} \right) \le 2 \exp \left( - \frac { (\eps c r)^2}{12} \right).
$$
By a union bound, a given large square is bad with probability at most
$$
2 \ell^2 \exp \left( - \frac { (\eps c r)^2}{12} \right) \le 2 \left( 2 \lg r \right)^2 \exp \left( - \frac { (\eps c r)^2}{12} \right) \le \frac {1}{\lg^4 r};
$$
both inequalities hold provided $C$ is large enough. (Note that $\ell \le 20 \lceil \lg r/20 \rceil \le 2 \lg r$, provided $r \ge C$.)

Now, the number of large squares that are bad can be stochastically bounded from above by the random variable $Y \sim \Bin(k^2, 1/\lg^4 r)$. By part (i), we may assume that, say, $r = O(\log n)$ and so, in particular, $r \lg r = o(\sqrt{n})$. Note that 
$$
\E[Y] = \frac {k^2}{\lg^4 r} \sim \frac {n}{(c r \lg r)^2 \lg^4 r} \le \frac {n}{3 r^2 \lg^5 r},
$$
provided $C$ is large enough. On the other hand, note that, say, $\E[Y] = \Omega( n / \log^3 n )$. Hence, it follows immediately from Chernoff's bound that 
$$
\Prob \left( Y \ge \frac {n}{r^2 \lg^5 r} \right) \le \Prob \left( Y \ge 2 \E[Y] \right) = \exp \left( - \Omega( n / \log^3 n ) \right) = o(1/\sqrt{n}).
$$
By the de-Poissonization argument explained above, the desired property holds for $\G(n,r)$ and the proof is finished.
\end{proof}

The next deterministic result shows that there exists an acquisition protocol that pushes weights from all vertices of each large good square into a single vertex.

\begin{theorem}\label{thm:good_squares} 
Fix $c = 1/10000$, $\eps = 1/100$, $n \in \N$, and radius $r=r_n \ge C$ for some large enough constant $C \in \R$. Consider any distribution of vertices that makes a given large square $\mathcal{S}$ good (with respect to $c$, $\eps$, $r$, and $n$). Finally, let $G$ be any geometric graph induced by vertices from $\mathcal{S}$ with radius $r$. Then, $a_t(G) = 1$.
\end{theorem}

Before we prove this theorem, let us state the following corollary that follows immediately from Theorems~\ref{thm:good_squares} and~\ref{thm:distr_good_squares}. 

\begin{corollary}\label{cor:upper_bound}
Suppose that $r = r_n$ is such that $r \lg r \le \sqrt{n}$ and $r \ge C$ for some sufficiently large $C \in \R$. Then, a.a.s.\ 
$$
a_t \left( \G(n,r) \right) = O \left( \frac {n}{(r \lg r)^{2}} \right).
$$
\end{corollary}
\begin{proof}
Let $c, \eps$ be fixed as in Theorem~\ref{thm:good_squares} and let $C = C(c, \eps)$ be the constant implied by Theorem~\ref{thm:distr_good_squares}. If $r \ge C \sqrt{\log n}$, then Theorem~\ref{thm:distr_good_squares}(i) implies that a.a.s.\ all large squares are good and so by Theorem~\ref{thm:good_squares} a.a.s.\ 
$$
a_t(\G(n,r)) \le k^2 = O \left( \frac {n}{(r \lg r)^2} \right).
$$

On the other hand, if $r \ge C$, then Theorem~\ref{thm:distr_good_squares}(ii) implies that a.a.s.\ at most $n/(r^2 \lg^5 r)$ large squares are bad. Clearly, each large bad square can be tessellated into $O(\lg^2 r)$ squares of side length $r / \sqrt{2}$, and so the graph $G$ induced by vertices of any large bad square satisfies $a_t(G) = O(\lg^2 r)$. This time we get that a.a.s.
$$
a_t(\G(n,r)) \le k^2 + \frac {n}{r^2 \lg^5 r} \cdot O(\lg^2 r) = O \left( \frac {n}{(r \lg r)^2} \right),
$$
and the proof of the corollary is finished.
\end{proof}

The only ranges of $r=r_n$ not covered by Corollary~\ref{cor:upper_bound} are when $r < C$ for $C$ as in the corollary or when $r \lg r > \sqrt{n}$. For the first case there is nothing to prove as the bound $O(n)$ trivially holds. The latter case follows immediately by monotonicity of $a_t(G)$.

Hence, it remains to prove Theorem~\ref{thm:good_squares}.

\begin{proof}[Proof of Theorem~\ref{thm:good_squares}]
Split $\mathcal{S}$ into four triangles using the two diagonals of $\mathcal{S}$. (Note that by property (a) of the distribution of the vertices, no vertex lies on the border of any triangle.) By symmetry, we may concentrate on the bottom triangle: the base of the triangle has length $\ell (yr)$ and the height is $\ell (yr)/2$. Since $\ell$ is divisible by $2$, the center of the large square is the corner of four small squares. Clearly, the number of small squares that are completely inside the triangle is $\ell^2 / 4 - \ell / 2$  (the total area of the triangle is $\ell^2/4$, and there are $\ell$ small squares only partially contained in this area, contributing a total area of $\ell/2$); on the other hand,  $\ell^2 / 4 + \ell / 2$ of them cover the triangle. Hence, since all small squares are good, the number of vertices $z$ that lie in the triangle is at most 
$$
\left( \frac {\ell^2}{4} + \frac{\ell}{2} \right) (1+\eps) (yr)^2 = \left( 1 + \frac{2}{\ell} \right) (1+\eps) \frac {(xr\lg r)^2}{4} \le (1+2\eps) \frac {(xr\lg r)^2}{4} =: z^+,
$$
provided that $C$ is large enough. Similarly, we get that $z \ge z^- := (1-2\eps) (xr\lg r)^2 / 4$.

Let $d$ be the smallest integer such that $2^d \ge z$. Since $z^- \le z \le z^+$, it follows that $d = \lg z + O(1) = 2 \lg r + 2 \lg \lg r + O(1)$. Observation~\ref{obs:trees} implies that there exists a rooted sub-tree $T$ of $\hat{T}_d$ on $z$ vertices with $a_t(T)=1$. Our goal is to show that $T$ can be embedded on the set of vertices that belong to the triangle with the root being the vertex closest (in Euclidean distance) to the apex of the triangle. If this can be done, then one can merge all the accumulated weights from the four triangles partitioning $\mathcal{S}$ into one of them and finish the proof: indeed, as the Euclidean distance from the closest vertex to the apex of the triangle is at most $\sqrt{5} y r \le \sqrt{5} c r \le r / 2$, the four roots induce a clique; see Figure~\ref{fig:root_of_5}.

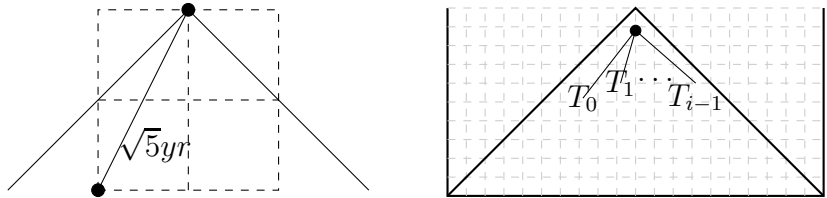
\begin{figure}[ht!]\centering
\begin{tikzpicture}[scale=1.2]
\filldraw (0,2) circle (2pt);
\filldraw (-1,0) circle (2pt);
\draw (-1,0) -- (0,2);
\draw (-0.37,0.5) node {$\sqrt{5}yr$};
\draw (-2,0) -- (0,2) -- (2,0);
\draw[dashed] (-1,0) rectangle (1,2);
\draw[dashed] (-1,1) -- (1,1);
\draw[dashed] (0,0) -- (0,2);
\end{tikzpicture}
\hskip 0.35 in
\begin{tikzpicture}[scale=1]
\draw[thick] (0,2.5) -- (0,0) -- (5,0) -- (5,2.5);
\foreach \x in {0.25,0.5,0.75,1,1.25,1.5,1.75,2,2.25,2.5,2.75,3,3.25,3.5,3.75,4,4.25,4.5,4.75}{\draw[dashed, black!20] (\x,0) -- (\x,2.5);}
\foreach \x in {0.25,0.5,0.75,1,1.25,1.5,1.75,2,2.25,2.5}{\draw[dashed, black!20] (0,\x) -- (5,\x);}
%\foreach \x in {0.25,0.5,0.75,1,1.25,1.5,1.75,2,2.25,2.5}{\draw[dashed, black!20] (0,5-\x) -- (5,5-\x);}
\draw[thick] (0,0) -- (2.5,2.5) -- (5,0);
%\draw[thick] (0,5) -- (2.5,2.5) -- (5,5);
\foreach \x in {(2.5,2.2)}{\filldraw \x circle (2pt);}
\foreach \x in {(1.8,1.3),(2.3,1.5),(3.3,1.5)}{\draw (2.5,2.2) -- \x;}
%\foreach \x in {(1.8,3.7),(2.3,4),(3.3,4)}{\draw (2.55,2.8) -- \x;}
%\foreach \x in {(3.7,1.8),(3.5,2.9),(3.5,3.3)}{\draw (2.75,2.45) -- \x;}
%\foreach \x in {(1.3,1.8),(1.5,2.9),(1.5,3.3)}{\draw (2.32,2.5) -- \x;}
\foreach \x in {(1.8,1.3)}{\draw \x node {$T_0$};}
\foreach \x in {(2.3,1.5)}{\draw \x node {$T_1$};}
\foreach \x in {(2.8,1.55)}{\draw[thick] \x node {$\dots$};}
\draw (3.3,1.3) node {$T_{i-1}$};
%\draw[white] (0,-0.3) node {.};
\end{tikzpicture}
\caption{On the left: there is a vertex in the triangle at distance at most $\sqrt{5}yr$ from the apex. On the right: in each triangle, we attempt to embed a tree that includes all vertices in the triangle. The four roots induce a clique, and so if such trees can be embedded, all weights in the square can be pushed onto a single vertex.}
\label{fig:root_of_5}
\end{figure}

We divide the triangle into $\ell / 20$ strips by introducing \textbf{auxiliary lines} $A_i$ ($i \in \{0, 1, \ldots, \ell/20\}$; recall that $\ell$ is divisible by 20), all of them are lines parallel to the base of the triangle.  $A_0$ is the line that passes through the apex of the triangle, $A_1$ is at distance $10yr$ from $A_0$, etc., $A_{\ell/20}$ coincides with the base of the triangle.
%\textbf{Strip} $j$ consists of the part of the triangle between auxiliary lines $A_{j-1}$ and $A_j$. We will also need a notion of regions. 
Note that there are exactly 10 strips of little squares between any two consecutive auxiliary lines $A_{j-1}$ and $A_j$. Any two points $a_1, a_2$ on the base of the triangle and a line $L$ parallel to the base induce an \textbf{auxiliary region}, a trapezoid with vertices $a_1, a_2$ and two vertices on $L$, the intersection of the line between the apex of the triangle and $a_1$ with $L$, and the intersection of the line between the apex of the triangle and $a_2$ with $L$, respectively. In particular, the triangle itself is a (degenerate) auxiliary region, induced by the two vertices from the base of the triangle and $A_0$. 

We will now give a recursive algorithm how to embed the tree $T$ on all $z$ vertices of the triangle. As already mentioned, we pick the vertex closest in Euclidean distance to the apex of the triangle and assign it to the root of $T$. Let $L_0$ be any line parallel to the base separating the vertex assigned to the root from other vertices that are not yet assigned to any vertex of $T$. (Note that by our assumption of the distribution of the vertices, there are no two vertices on any line parallel to the base.) This will be a typical situation that we have to deal with, in a recursive fashion. Suppose thus that we are given a line $L_{i-1}$ parallel to the base such that vertices above $L_{i-1}$ are already assigned to vertices in $T$, and vertices below $L_{i-1}$ that belong to the auxiliary region $\mathcal{Q}$ we currently deal with are not yet assigned to vertices in $T$. We will always keep the property that $\mathcal{Q}$ contains exactly the number of vertices we need to assign to some part of the tree $T$; these vertices induce a family of rooted trees in $T$, with roots that are at graph distance $i$ from the root of $T$. Denote by $Q_i$ and $R_i$ the number of vertices that belong to $\mathcal{Q}$ and, respectively, to the part of $\mathcal{Q}$ above $A_i$; see Figure~\ref{fig:5}.

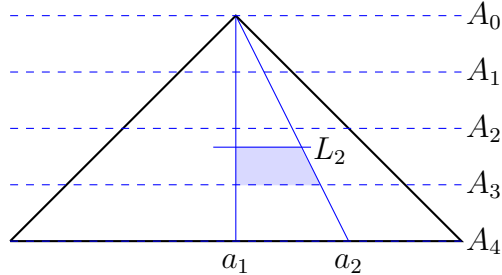
\begin{figure}[ht!]\centering
\hskip 0.5 in
\begin{tikzpicture}
\filldraw[blue!15] (3,0.75) -- (4.12,0.75) -- (3.85,1.25)--(3,1.25);
\draw[blue] (3,3) -- (3,0);
\draw[blue] (3,3) -- (4.5,0);
\draw[thick] (0,0) -- (3,3) -- (6,0) -- (0,0);
\foreach \x in {0,1,2,3,4}{\draw (6.3,3-\x*0.75) node {$A_\x$};\draw[blue, dashed] (0,3-\x*0.75) -- (6,3-\x*0.75);}
\draw (3,-0.3) node {$a_1$};
\draw (4.5,-0.3) node {$a_2$};
\draw[blue] (2.7,1.25) -- (4,1.25);
\draw (4.25,1.2) node {$L_2$};
\end{tikzpicture}
\caption{The number of vertices in the shaded region is $R_3$, the number of vertices in the trapezoid determined by $L_2$, the base of the triangle, and the two blue sides of the triangle associated with $\mathcal{Q}$ is $Q_3$.\label{fig:5}}
\end{figure}

Let $a_1$ and $a_2$ be the two corners of $\mathcal{Q}$ that belong to the base of the triangle. Let $b_1$ and $b_2$ be the intersection points of $A_i$ with the line going through the apex and $a_1$, and with the line going through the apex and $a_2$, respectively; see Figure~\ref{fig:6}. If the Euclidean distance between $b_1$ and $b_2$ is more than $r/3$, then we split $\mathcal{Q}$ into two auxiliary regions (the first one induced by $b_1$ and some $b$ on $A_i$, the other one induced by $b$ and $b_2$; in both situations the auxiliary line $L_{i-1}$ is used), where $b$ is chosen in such a way that $Q_i$ vertices are partitioned into two families of rooted trees in $T$ as evenly as possible. Observe that it is possible to split $\mathcal{Q}$ in such a way so that both auxiliary sub-regions contain at least $Q_i/4$ vertices; indeed, one can order the family of rooted trees according to their sizes and then notice that adding one rooted tree to one of the auxiliary sub-regions obtained after splitting can increase the total number of vertices there by a multiplicative factor of at most 2. (Note that by property (c) of the distribution of the vertices, we can perform a split so that no vertex belongs to the border of any resulting auxiliary region.) We stop the algorithm prematurely if the Euclidean distance between $b_1$ and $b$ (or between $b$ and $b_2$) is less than $r/20$ or more than $r/3$ (\textbf{Error 1} is reported). If everything goes well, we deal with each auxiliary region recursively (we update $Q_i$ and $R_i$, and all lines defining the auxiliary region).

\begin{figure}[ht!]\centering
\begin{tikzpicture}[scale=1.5]
\draw (2,0) -- (5,0);
\draw[blue,dashed] (2,0.75) -- (5,0.75);
\draw (3,-0.2) node {$a_1$};
\draw (4.5,-0.2) node {$a_2$};
\draw[blue] (2.7,1.25) -- (4,1.25);
\draw (4.25,1.3) node {$L_2$};
\draw (5.25,0.75) node {$A_3$};
\draw[blue] (3,0) -- (3,2);
\draw[blue] (4.5,0) -- (3.5,2);
\filldraw (3,0.75) circle (1.5pt);
\filldraw (4.15,0.75) circle (1.5pt);
\draw (2.8,0.6) node {$b_1$};
\draw (4.45,0.6) node {$b_2$};
\filldraw (3.65,0.75) circle (1.5pt);
\draw (3.55, 0.55) node {$b$}; 
\draw (3.85,0) -- (3.5,1.25);
%\draw (3.85,-0.2) node {$a$};
\end{tikzpicture}
\caption{If the Euclidean distance between $b_1$ and $b_2$ is more than $r/3$, we split the region into two regions.} \label{fig:6}
\end{figure}
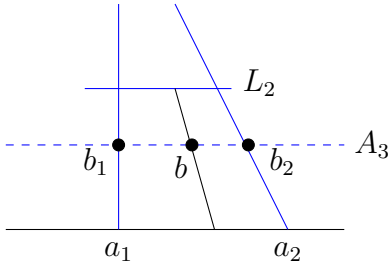

Now, we want to assign all roots from the family of rooted trees (recall that they are at level $i$ of $T$) to vertices of $\mathcal{Q}$ above $A_i$. If there are more than $R_i$ vertices on level $i$ in $T$, then stop the algorithm prematurely (\textbf{Error 2} is reported). In fact, we typically only need to embed a small portion of the vertices of level $i$, but we nevertheless stop prematurely if $R_i$ is smaller than the total number of vertices at level $i$ in the tree. Otherwise, we first assign all roots of the family of rooted trees we deal with. Then, we order the trees rooted at them according to their sizes (in non-decreasing order), and keep adding whole rooted trees, as long as the total number of vertices added is at most $R_i$ (see Figure~\ref{fig:tree}). By the same argument as before, we are guaranteed that at least $R_i/2$ vertices are assigned to the corresponding vertices of $T$. Clearly, if $i=\ell/20$, we are able to fit all rooted trees, and so all $R_i$ (which is equal to $Q_i$ in this case) vertices are dealt with. On the other hand, that is, as long as $i < \ell/20$, we introduce any line $L_i$, parallel to the base, that separates vertices of $\mathcal{Q}$ that are assigned (that are above the line) from those that are still not assigned to any vertex in $T$ (below the line). (As usual, by property (b) of the distribution of the vertices, we can do it so that no vertex lies on $L_i$.) We continue recursively with the new auxiliary region below $L_i$ and the new family of rooted trees consisting of all the branches that are not assigned to any vertices; see Figure~\ref{fig:tree}. Note that the line $L_i$ depends on $\mathcal{Q}$, and  different auxiliary regions corresponding to embedding vertices of $T$ of the same level might have a different line $L_i$. We will show below that these lines will all be close to $A_i$. 

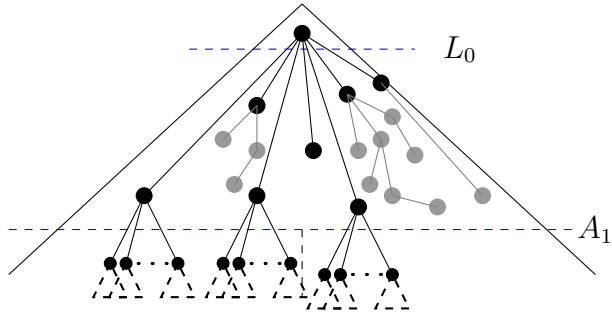
\begin{figure}[ht!]\centering
\begin{tikzpicture}[scale=3]
\foreach \x in {(5.25,5.35), (5.4,5.5), (5.5,5.33), (5.8,5.15), (4.8,5.55), (4.65,5.4), (4.8,5.35), (4.7,5.2), (5.35,5.4), (5.3,5.2), (5.4,5.15), (5.6,5.1)}{
\filldraw[black!40] \x circle (1pt);
}
\foreach \x in {(5.2,5.6), (5.35,5.65), (4.8,5.55)}{
\filldraw \x circle (1pt);}
\draw[dashed,blue] (4.5,5.8) -- (5.5,5.8);
\draw (5.7,5.8) node {$L_0$};
\draw (5,5.87) -- (5.2,5.6);
\draw (5,5.87) -- (4.8,5.55);
\draw[black!50] (4.65,5.4) -- (4.8,5.55);
\draw[black!50] (4.8,5.35) -- (4.8,5.55);
\draw[black!50] (4.8,5.35) -- (4.7,5.2);
\draw (5,5.87) -- (5.35,5.65);
\draw[black!50] (5.8,5.15) -- (5.35,5.65);
\draw[black!50] (5.2,5.6) -- (5.25,5.35);
\draw[black!50] (5.2,5.6) -- (5.4,5.5);
\draw[black!50] (5.3,5.2) -- (5.35,5.4);
\draw[black!50] (5.4,5.15) -- (5.35,5.4);
\draw[black!50] (5.4,5.15) -- (5.6,5.1);
\draw[black!50] (5.5,5.33) -- (5.4,5.5);
\draw[black!50] (5.35,5.4) -- (5.2,5.6);
\draw (6.3,4.8) -- (5,6) -- (3.7,4.8);
\draw[dashed,blue] (3.7,5) -- (6.2,5);
\draw (6.3,5) node {$A_1$};
\draw[dashed] (5,5) -- (5,4.7);
\filldraw (5,5.87) circle (1pt);
\foreach \x in {(5.25,5.1), (4.8,5.15), (5.05,5.35), (4.3,5.15)}{
\filldraw \x circle (1pt);
\draw (5,5.87) -- \x;}
\foreach \x in {(5.1,4.8),(5.17,4.8),(5.4,4.8)}{\filldraw \x circle (0.75pt); \draw \x -- (5.25,5.1);\draw (5.3,4.8) node {$\dots$};}
\foreach \x in {(4.65,4.85),(4.72,4.85),(4.95,4.85)}{\filldraw \x circle (0.75pt); \draw \x -- (4.8,5.15);\draw (4.85,4.85) node {$\dots$};}
\foreach \x in {(4.15,4.85),(4.22,4.85),(4.45,4.85)}{\filldraw \x circle (0.75pt); \draw \x -- (4.3,5.15);\draw (4.35,4.85) node {$\dots$};}
\foreach \x in {4.45,4.95,4.22,4.15,4.65,4.72}{\foreach \y in {4.85}{\draw[dashed,thick] (\x,\y) -- (\x-0.075,\y-0.15) -- (\x+0.075,\y-0.15) -- (\x,\y);}}
\foreach \x in {5.4,5.1,5.17}{\foreach \y in {4.8}{\draw[dashed,thick] (\x,\y) -- (\x-0.075,\y-0.15) -- (\x+0.075,\y-0.15) -- (\x,\y);}}
\end{tikzpicture}
\caption{Vertices from layer 1 in the tree are assigned to vertices in $Q$. We assign the rest of the vertices in $\mathcal{Q}$ by embedding entire branches of the tree, as long as the number of vertices assigned is at most $R_i$ (in grey). The remaining branches become roots for the next iteration. }\label{fig:tree}
\end{figure}

Finally, if at some point of this process two vertices in $\mathcal{Q}$ are assigned to two adjacent vertices in $T$ that are at Euclidean distance more than $r$, then we clearly have to stop the algorithm prematurely (\textbf{Error 3} is reported). 

It remains to argue that we never stop the algorithm prematurely as this implies that $T$ is embedded on the vertices inside the triangle. Let us deal first with Error~2, then with Error~3, leaving Error~1 for the end. 

\medskip

\textbf{Error 2}---level $i$ in $T$ contains more vertices than are available in $\mathcal{Q}$ above $A_i$ (that is, more than $R_i$): First, let us observe that for $i \in \{1, 2, \ldots, 50\}$, the auxiliary line $A_i$ intersects the triangle so that the Euclidean distance between the two points on the sides of the triangle under consideration intersecting with $A_i$ is $(20yr) i \le (20 c r) i = r i /500 \le r/3$. Hence, splitting of auxiliary regions cannot happen during the first 50 rounds. On the other hand, for $i \in \{1, 2, \ldots, 50\}$ we have $(20yr) i \ge (10 c r) i = r i /1000$. Let us then concentrate on any $i \in \{51, 52, \ldots, \ell/20\}$. We show, inductively, that when dealing with line $A_i$, the two corresponding points $b_1$ and $b_2$ are at distance at least $r/20$. The claim is true for $A_{50}$ as argued above. Suppose then that the claim holds for $A_{i-1}$ for some $i \in \{51, 52, \ldots, \ell/20\}$. If $\mathcal{Q}$ is split into two auxiliary regions, then the claim holds for $A_i$ unless \textbf{Error~1} is reported. On the other hand, if no splitting is performed, then the Euclidean distance between the two corresponding points can only increase, and so the claim clearly holds for $A_i$. This implies, in particular, that $\mathcal{Q}$ contains at least one small square, and thus $R_i \ge (1-\eps)(yr)^2 \ge (1-\eps)(cr/2)^2 \ge r^2 10^{-9}$. On the other hand, since
$$
i \le \frac {\ell}{20} = \frac {(x r \lg r)/(yr)}{20} \le \frac {c \lg r}{10 c} = \frac {\lg r}{10},
$$
we get from Observation~\ref{obs:trees} that the number of vertices on level $i$ in $T$ is at most
$$
\binom{d}{i} \le \binom{2 \lg r + 2 \lg \lg r + O(1)}{\lg r / 10 } \le \binom{ 3 \lg r }{  \lg r / 10 } \le (30e)^{\lg r / 10} \le 2^{7 \lg r / 10} < r^2 10^{-9},
$$
provided that $C$ is large enough. Hence, this error never occurs.

\medskip

\textbf{Error 3}---two vertices assigned to adjacent vertices in $T$ are at distance more than $r$: It follows from the definition of $L_i$ that for any $i \in \{ 1, 2, \ldots, \ell/20\}$, $L_i$ lies above the auxiliary line $A_i$ ($L_0$ is exceptional and lies slightly below $A_0$). We are going to argue that $L_i$ is relatively close to $A_i$. 

\smallskip

\noindent \emph{Claim: For any $i \in \{ 4, 5, \ldots, \ell/20\}$, $L_i$ lies below the auxiliary line $A_{i-4}$.}

\smallskip

We will be done once the claim is proved as it implies that we never connect vertices by an edge that are at Euclidean distance more than $50 (yr) + r/3 \le 50 c r + r/3 < r$. Indeed, vertices that need to be connected by an edge must lie in the part of $\mathcal{Q}$ between $L_{i-1}$ and $A_i$. The Euclidean distance between $L_{i-1}$ and $A_i$ is at most $50(yr)$ and the intersection of $A_i$ and $\mathcal{Q}$ is at most $r/3$; see Figure~\ref{fig:connecting_Vertical}.

\smallskip

\noindent \emph{Proof of the Claim:} For a contradiction, suppose that there exists $i$ such that $L_i$ lies above $A_{i-4}$ and consider the smallest such $i$. Hence, $L_{i-1}$ lies below $A_{i-5}$. Let $\mathcal{Q}_1$ be the part of $\mathcal{Q}$ that lies between $L_{i-1}$ and $L_i$, and recall that $\mathcal{Q}_1$ contains at least $R_i/2$ vertices. Similarly, let $\mathcal{Q}_2$ be the part between $L_{i-1}$  and $A_i$ and recall that $\mathcal{Q}_2$ contains precisely $R_i$ vertices. 
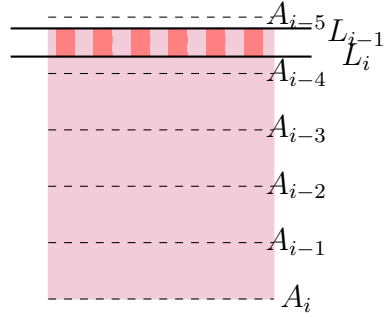
\begin{figure}[ht!]\centering
\begin{tikzpicture}
\filldraw[purple!20] (0,0) -- (3,0) -- (3,4.8*0.75) -- (0,4.8*0.75) -- (0,0);
\foreach \x in {0.1,0.6,1.1,1.6,2.1,2.6}{\filldraw[red!50, dashed] (0+\x,4.3*0.75) -- (0.25+\x,4.3*0.75) -- (0.25+\x,4.8*0.75) -- (0+\x,4.8*0.75);}
\foreach \i in {1,2,3,4, 5}{
\draw[dashed] (0,0.75*\i) -- (3,0.75*\i);
\draw (3.3,0.75*\i) node {$A_{i-\i}$};
}
\draw[dashed] (0,0) -- (3,0);
\draw (3.3,0) node {$A_{i}$};
\draw[thick] (-0.5,0.75*4.3) -- (3.5,0.75*4.3);
\draw (4.1,0.75*4.3) node {$L_{i}$};
\draw[thick] (-0.5,0.75*4.8) -- (3.5,0.75*4.8);
\draw (4.1,0.75*4.7) node {$L_{i-1}$};
\end{tikzpicture}
\caption{All vertices that need to be connected by an edge must be in the part of $\mathcal{Q}$ between $L_{i-1}$ and $A_i$.} \label{fig:connecting_Vertical}
\end{figure}
The fact that the area of $\mathcal{Q}_1$ is at least five times smaller than the one of $\mathcal{Q}_2$ but it contains at least half of vertices will lead us to the desired contradiction. Recall that the length of the intersection of $A_{i-4}$ with the triangle is $s \ge r/20-80(yr) \ge r/20-r/125 \ge r/25$. Hence, the number of small squares covering $\mathcal{Q}_1$ is at most $10(u+2)$, where $u=\lceil s/(yr) \rceil \ge 400$. The number of vertices in $\mathcal{Q}_1$ is then at most $10(u+2)(1+\eps)$, and so $R_i \le 20(u+2)(1+\eps)$. On the other hand, the number of small squares that are completely contained in $\mathcal{Q}_2 \setminus \mathcal{Q}_1$ is at least $40(u-2)$, and so $R_i$ is at least $40(u-2)(1-\eps)$. The contradiction follows, since $20(u+2)(1+\eps)<40(u-2)(1-\eps)$ for any $u \ge 400$.

\medskip

\textbf{Error 1}---the Euclidean distance between $b_1$ and $b$ (or $b$ and $b_2$) is either less than $r/20$ or more than $r/3$: Suppose that we partition $\mathcal{Q}$ containing $Q_i$ vertices into $\mathcal{Q}_1$ and $\mathcal{Q}_2$, where $\mathcal{Q}_1$ is the part of $\mathcal{Q}$ induced by $b_1$ and $b$, all the way down to $A_{\ell/20}$. Recall that $\mathcal{Q}_1$ contains at least $Q_i/4$ vertices, and that the Euclidean distance between $b_1$ and $b_2$ is more than $r/3$ (since we performed splitting). 

\begin{figure}[ht!]\centering
\begin{tikzpicture}
\draw[thick] (0,0) -- (5,0);
\draw[thick] (0,0) -- (2.5,5);
\draw[thick] (5,5) -- (5,0);
\draw[thick] (2,0) -- (3.5,5); 
\draw[thick, blue] (1,3.1) -- (5,3.1);
\draw[blue] (0.6,3.2) node {$L_{i-1}$};
\filldraw (1.55,3.1) circle (2pt);
\draw (1.25,3.35) node {$d_1$};
\draw (5.3,3.35) node {$d_2$};
\draw (2.85,3.35) node {$d$};
\filldraw (2.9,3.1) circle (2pt);
\filldraw (5,3.1) circle (2pt);
\filldraw (5,0) circle (2pt);
\filldraw (0,0) circle (2pt);
\filldraw (2,0) circle (2pt);
\draw (1.8,2) node {$\mathcal{Q}_1$};
\draw (-0.3,0.3) node {$b_1$};
\draw (5.3,0.3) node {$b_2$};
\draw (2.3,0.3) node {$b$};
\draw[black!60] (5,2.4) -- (1.2,2.4);
\draw (0.75,2.4) node {$A_{i-1}$};
\draw [brace] (1.7,3.25) -- (4.9,3.25);
\draw (3.35,3.5) node {$s$};
\draw[white] (0,-1.1) node {.};
\end{tikzpicture}\hskip 0.5 in
\begin{tikzpicture}
\filldraw[red!30] (1.5,3) rectangle (5,3.25);
\filldraw[red!30] (5,0) rectangle (4.75,3.25);
\foreach \x/\y in {0/0,0/0.25,0.25/0.5,0.25/0.75,0.5/1,0.5/1.25,0.75/1.5,0.75/1.75,1/2,1/2.25,1.25/2.5,1.25/2.75}{\filldraw[red!30] (0+\x,0+\y) rectangle (0.25+\x,0.25+\y);}
\draw[thick] (0,0) -- (5,0);
\foreach \x in {0.25,0.5,0.75,1,1.25,1.5,1.75,2,2.25,2.5,2.75,3,3.25,3.5,3.75,4,4.25,4.5,4.75}{\draw[dashed, black!20] (\x,0) -- (\x,5);}
\foreach \x in {0.25,0.5,0.75,1,1.25,1.5,1.75,2,2.25,2.5}{\draw[dashed, black!20] (0,\x) -- (5,\x);
\draw[dashed, black!20] (0,\x+2.5) -- (5,\x+2.5);}
\draw[thick] (0,0) -- (2.5,5);
\draw[thick] (5,5) -- (5,0);
\draw[thick] (2,0) -- (3.5,5); 
\draw[thick, blue] (0,3.1) -- (5,3.1);
\draw[blue] (-0.4,3.2) node {$L_{i-1}$};
\draw [bracem] (0,-0.1) -- (1.9,-0.1);
\draw [bracem] (0,-0.6) -- (4.9,-0.6);
\draw (0.9,-0.4) node {$<r/3$}; 
\draw (2.5,-1) node {$>r/20$}; 
\draw[black!60] (5,2.4) -- (1,2.4);
\draw (-0.3,2.4) node {$A_{i-1}$};
\end{tikzpicture}
\caption{On the left: definitions of points and regions used in \textbf{Error~1}. On the right: illustration of the squares in case \textbf{Error 1} occurs because the Euclidean distance between $b_1$ and $b$ is less than $r/20$.} \label{FigureError1}
\end{figure}
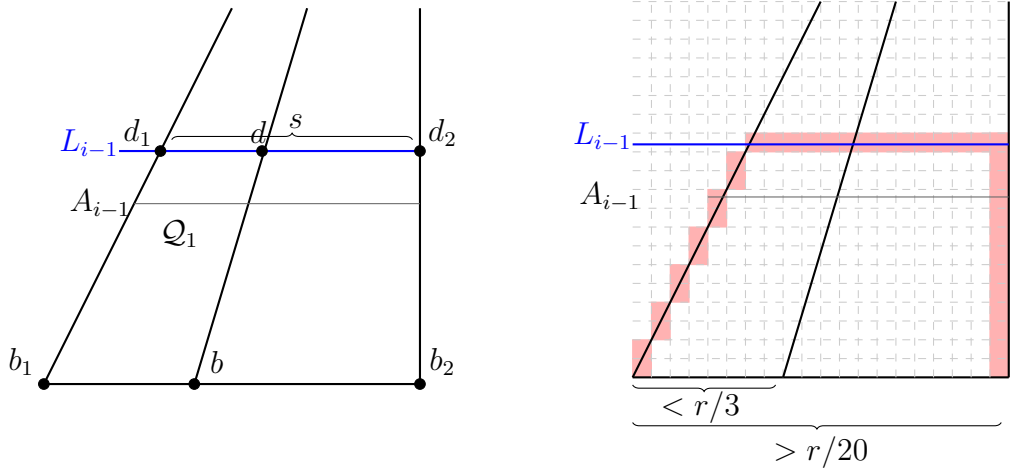

Suppose that \textbf{Error 1} occurs because the Euclidean distance between $b_1$ and $b$ is less than $r/20$. Exactly the same argument can be applied to the case when the Euclidean distance between $b$ and $b_2$ is less than $r/20$. Let $d_1$, $d$, and $d_2$ be the three points of intersection of the line $L_{i-1}$ with the lines going between the apex of the triangle and $b_1$, $b$ and, respectively, $b_2$; see Figure~\ref{FigureError1}. Note that the Euclidean distance between $d_1$ and $d$ is less than $r/20$ and, since by the claim $L_{i-1}$ lies below $A_{i-5}$, the Euclidean distance between $d_1$ and $d_2$, denoted by $s$, satisfies $s > r/3 - 100(yr) \ge 3 r /10$. As the corresponding triangles are similar, the length of the intersection of each horizontal line between $L_{i-1}$ and $A_{\ell/20}$ inside $\mathcal{Q}_1$ is at most a factor of $(r/20)/(r/3) = 3/20$ of the total length of the intersection of the line with the triangle. Hence, the area of $\mathcal{Q}_1$ is by a multiplicative factor of at most $3/20$ smaller than the area of $\mathcal{Q}$, which will be denoted by $A$. 

Arguing as in the previous error, the area of small squares completely contained in $\mathcal{Q}$ is at least $A \cdot \frac {u-2}{u} \cdot \frac {10}{11} \ge 0.9 A$ ($u = \lceil s/(yr) \rceil \ge \lceil (3r/10)/(yr) \rceil \ge 3000$). Indeed, since $L_{i-1}$ might cross small squares, the first row of small squares that intersects $\mathcal{Q}$ might be completely lost, giving an additional factor of $10/11$ (note that there are at least $10$ complete rows between $A_{i-1}$ and $A_i$). It follows that $Q_i \ge 0.9 A (1-\eps) > 0.89 A$.

On the other hand, the area of small squares having non-empty intersection with $\mathcal{Q}_1$ is at most $A \cdot \frac{3}{20} \cdot \frac {u'+3}{u'} \cdot \frac {11}{10} < 0.17 A$ ($u'=\lceil (3s/20)/(yr) \rceil \ge 450$). Hence, the total number of vertices in $\mathcal{Q}_1$ is at most $0.17 A (1+\eps) < 0.18 A$. This time we get $Q_i \le 4 \cdot 0.18 A = 0.72 A$, and the desired contradiction occurs. 

Finally, let us note that \textbf{Error 1} cannot occur because the Euclidean distance between $b_1$ and $b$ is larger than $r/3$ (provided that the distances between $b_1$ and $b$ as well as between $b$ and $b_2$ are at least $r/20$). Since we consider the smallest $i$ for which such error occurred, the length of the intersection of $A_i$ with $\mathcal{Q}$ is at most $r/3+20(yr)$ and so the Euclidean distance between $b_1$ and $b$ is at most $r/3+20(yr)-r/20 < r/3$. The same argument shows that the Euclidean distance between $b$ and $b_2$ cannot be larger than $r/3$.
\end{proof}

\section{Concluding remarks}
 The proof of the lower bound can be easily generalized to show that for any fixed dimension $d$ and sufficiently large radius $r$, $a_t(G)=\Omega(n/(r \lg r)^d)$. For $d=1$, it is also easy to get the matching upper bound $a_t(G)=O(n/(r \lg r))$. It is natural to conjecture that for $d \ge 3$ the proof of the upper bound can also be adapted to show $a_t(G)=O(n/(r \lg r)^d)$, but in order not to make the paper too technical, we opted for not pursuing further this approach.

\end{document}